%% file: czedli-dekany-ozsvart-szakacs-udvari_.tex
\documentclass{amsart}
\usepackage{amssymb,latexsym}
\usepackage{amsmath}
\usepackage[dvipdfm]{graphicx}
\usepackage{color}
\usepackage{enumerate}
\newcommand \url [1] {\texttt{#1}}
\input{slimsemmacros.tex}

\newcommand \inv [1] { \textup{inv} (#1) }
\newcommand \rspace {\kern 0.8pt}

\begin{document}
\title
{On the number of slim, semimodular lattices}
\author[G.\ Cz\'edli]{G\'abor Cz\'edli}
\email{czedli@math.u-szeged.hu}
\urladdr{http://www.math.u-szeged.hu/$\sim$czedli/}
\address{University of Szeged\\Bolyai Institute\\
Szeged, Aradi v\'ertan\'uk tere 1\\HUNGARY 6720}


%
\author[T.\ D\'ek\'any]{Tam\'as D\'ek\'any}
\email{dekany@math.u-szeged.hu}
\urladdr{http://www.math.u-szeged.hu/$\sim$dekany/}
\address{University of Szeged\\Bolyai Institute\\
Szeged, Aradi v\'ertan\'uk tere 1\\HUNGARY 6720}


\author[L.\ Ozsv\'art]{L\'aszl\'o Ozsv\'art}
\email{ozsvartl@math.u-szeged.hu}
\urladdr{http://www.math.u-szeged.hu/$\sim$ozsvart/}
\address{University of Szeged\\Bolyai Institute\\
Szeged, Aradi v\'ertan\'uk tere 1\\HUNGARY 6720}


\author[N.\ Szak\'acs]{N\'ora Szak\'acs}
\email{szakacs@math.u-szeged.hu}
\urladdr{http://www.math.u-szeged.hu/$\sim$szakacs/}
\address{University of Szeged\\Bolyai Institute\\
Szeged, Aradi v\'ertan\'uk tere 1\\HUNGARY 6720}


\author[B.\ Udvari]{Bal\'azs Udvari}
\email{udvarib@math.u-szeged.hu}
\urladdr{http://www.math.u-szeged.hu/$\sim$udvari/}
\address{University of Szeged\\Bolyai Institute\\
Szeged, Aradi v\'ertan\'uk tere 1\\HUNGARY 6720}


\thanks{This research was supported by the NFSR of Hungary (OTKA), grant numbers   K77432 and K83219, and by T\'AMOP-4.2.1/B-09/1/KONV-2010-0005 and  
T\'AMOP-4.2.2/B-10/1-2010-0012}



\date{June 5, 2012}

\begin{abstract} A lattice $L$ is \emph{slim} if it is finite and the set of its join-irreducible elements contains no three-element antichain. Slim, semimodular 
lattices were previously characterized by G.~Cz\'edli and E.\,T.~Schmidt~\cite{czgschperm}
as the duals of the lattices consisting of the intersections of the members of two composition series in a group.
Our main result determines the  number of (isomorphism classes of) these lattices of a given size in a recursive way. The corresponding planar diagrams, up to similarity, are also enumerated. 
We prove that the number of diagrams of slim, distributive lattices of a given length $n$ is the $n$-th Catalan number. Beside  lattice theory, the paper includes  some combinatorial arguments on permutations and their inversions.
\end{abstract}

\maketitle

\section{Introduction and target}\label{section:intRo}

The well-known concept of a composition series in  a group goes back to \'Evariste Galois (1831), see J.\,J.\ Rotman~\cite[Thm.\ 5.9]{rRotman}. The Jordan-H\"older theorem, stating that any two  composition series of a finite group have the same length, was also proved in the nineteenth century, see C.\ Jordan~\cite{rJordan} 
 and O.\ H\"older~\cite{rHolder}.
Let 
\begin{equation}\label{tWocomPsezwsS}
\begin{aligned}
\vakvec H&:\quad G= H_0 \triangleright H_1 \triangleright \cdots \triangleright H_h =\set 1\text{ and} \cr
\vakvec K&:\quad G=K_0 \triangleright K_1  \triangleright \cdots \triangleright K_h =\set1
\end{aligned}
\end{equation}
be  composition series of a group $G$. Consider the following structure: 
\[
\bigl(\bigset{H_i\cap K_j: i,j\in\set{0,\ldots,h}},\subseteq \bigr)\text.
\]
It is a lattice, a so-called \emph{composition series lattice}. The study of these lattices led  G.~Gr\"atzer and J.\,B.  Nation~\cite{gratzernation} and 
G.~Cz\'edli and  E.\,T.~Schmidt~\cite{r:czg-sch-JH}
to recent generalizations of the Jordan-H\"older theorem. 
In order to give an abstract characterization of these lattices,  G.~Cz\'edli and  E.\,T.~ Schmidt~\cite{czgschperm}  proved that composition series lattices are exactly the duals of slim, semimodular lattices, to be defined later. (See also \cite{r:czgolub} for a more direct approach to this result.)

Here we continue the investigations started by G. Cz\'edli, L. Ozsv\'art, and B. Udvari~\cite{r:czgolub}.
Our main goal is determine the number $\numssl n$ of slim, semimodular lattices (equivalently, composition series lattices) of a given size  $n$. 
Isomorphic lattices are, of course,  counted only once. These lattices of a given length were previously enumerated in \cite{r:czgolub}; however, the present task is subtler.  Since slim lattices are planar by G.~Cz\'edli and E.\,T.~Schmidt~\cite[Lemma 2.2]{r:czg-sch-JH}, we are also interested in the number of their planar diagrams. 
Due to the fact that we count \emph{specific} lattices, we  give a recursive description for   $\numssl n$ that is  far more efficient than the best known way to count \emph{all} finite lattices of a given size~$n$; see  J.\ Heitzig and J.\ Reinhold~\cite{r:heitzigreinhold} and the references therein.  
We also enumerate the planar diagrams of slim, semimodular lattices of size $n$, up to similarity to be defined later.

\subsection*{Outline} Section~\ref{lattsecT} belongs to Lattice Theory. After presenting the necessary concepts, it reduces the targeted problems to combinatorial problems on permutations. Section~\ref{countsection} belongs to Combinatorics. Theorem~\ref{thmmaIn} determines the number of slim, semimodular lattices consisting of $n$ elements. Proposition~\ref{prodzSkW} gives the number of the planar diagrams of slim, semimodular lattices of size $n$ such that similar diagrams are counted only once. The number of planar diagrams of slim, distributive lattices of a given length is proved to be a Catalan number in Proposition~\ref{pRdDlNgTnh}.

\section{From slim, semimodular lattices to permutations}\label{lattsecT}
\subsection*{An overview of slim,  semimodular lattices} 
Since G.~Cz\'edli and G.~Gr\"atzer~\cite[Theorem 1-3.5]{ggwltsta} is not generally available when writing this paper, we usually recall the necessary prerequisites from the original sources. 
All lattices occurring in this paper are assumed to be finite. 
The notation is taken from G.~Gr\"atzer~\cite{r:Gr-LTFound}. In particular,  the set of non-zero join-irreducible elements of a lattice $L$ is denoted by $\Jir L$. If $\Jir L$ is a union of two chains (equivalently, if $\Jir L$ contains no three-element antichain), then $L$ is called a \emph{slim} lattice. 
Slim lattices are \emph{planar} by G.~Cz\'edli and E.\,T.~Schmidt~\cite[Lemma 2.2]{r:czg-sch-JH}. That is, they possess  planar diagrams. Let $D_1$ and $D_2$ be planar lattice diagrams. A bijection $\phi\colon D_1\to D_2$ is a \emph{similarity map} if it is a lattice isomorphism and for all $x, y, z \in D_1$ such that $x\prec y$ and $x \prec  z$, $y$ is to the left of $z$ if{}f $\phi(y)$ is to the left of $\phi(z)$. 
Following D.~Kelly and I.~Rival~\cite[p.\, 640]{kellyrival}, we say that $D_1$ and $D_2$ are \emph{similar lattice diagrams}  if there exists a
similarity map $D_1 \to D_2$. We always consider and count planar diagrams up to similarity. Also, we consider only planar diagrams. A diagram is slim if it represents a slim lattice; other lattice properties apply for diagrams analogously. For example, a diagram is \emph{semimodular} if so is the corresponding lattice $L$; that is, if for all $x,y,z\in L$ such that $x\preceq y$, the covering or equal relation $x\vee z\preceq y\vee z$ holds.
\begin{figure}
\centerline
{\includegraphics[scale=1.0]{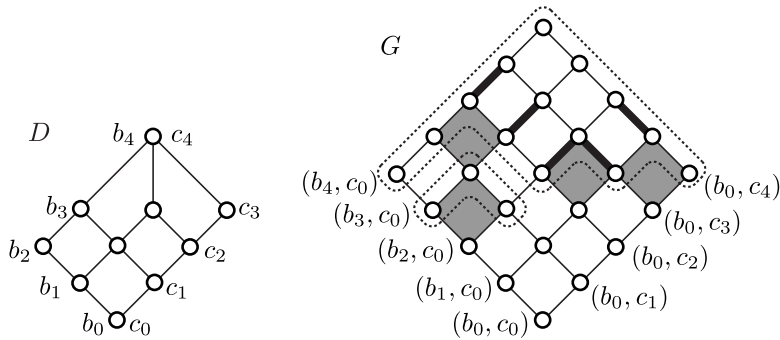}}
\caption{A diagram $D$ and the corresponding grid diagram $G$  \label{fig:egy}}
\end{figure}

Let $D$ be a planar diagram of a slim lattice $L$ of length $h$. Note that $L$ may have several non-similar diagrams since we can reflect $D$ (or certain intervals of $D$)  vertically. The left boundary chain of $D$ is denoted by $\leftb D$, while $\rightb D$ stands for its right boundary chain. These chains are maximal chains in $L$, and both are of length $h$ by semimodularity. 
So we can write
\begin{equation}\label{lrlchNps}
\begin{aligned}
\leftb D&=\set{0=b_0\prec b_1\prec\cdots\prec b_h}\text{ and }\cr
\rightb D&=\set{0=c_0\prec c_1\prec\cdots\prec c_h}\text.
\end{aligned}
\end{equation}

\subsection*{The permutation of a slim,  semimodular lattice} 
The present paper is based on the fundamental connection between planar, slim, semimodular diagrams and permutations. In  this and the next subsections, we recall and develop the details of this connection in a way that fits \cite{r:czgolub}, where the enumerative investigations of slim,  semimodular lattices start. %
The following statement is a particular case of G.~Cz\'edli and G.~Gr\"atzer~\cite[Theorem 1-3.5]{ggwltsta}; it can also be extracted from G.~Cz\'edli and E.\,T.~Schmidt~\cite[Lemmas 6 and 7]{r:czg-sch-visual} combined with \cite[Proof 4.7]{r:czg-sch-patch}.
\begin{figure}
\centerline
{\includegraphics[scale=1.0]{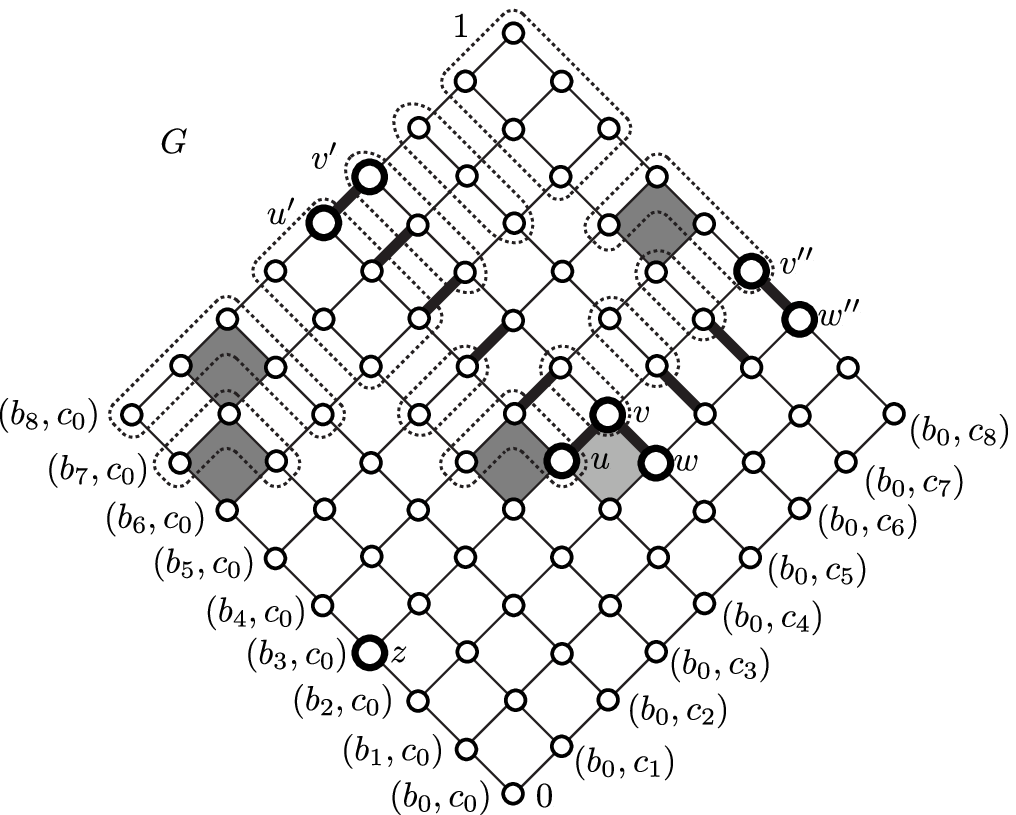}}
\caption{A grid  \label{fig:ket}}
\end{figure}

\begin{lemma}\label{simLmaPlD}
Assume that $D$ and $E$ are planar diagrams of a slim, semimodular lattice  $L$. Then $D$ is similar to $E$ if{f} $\leftb D=\leftb E$ if{f} $\rightb D=\rightb E$.
\end{lemma}

Next, with $D$ as above, consider the diagram $G$ of the (slim, distributive) lattice $\leftb D\times \rightb D$ such that $\leftb D\times\set 0\subseteq \leftb G$  and $\set0\times \rightb D\subseteq \rightb G$. Then $G$ is determined up to similarity by Lemma~\ref{simLmaPlD}, and it is called the \emph{grid diagram} associated with $D$; see Figure~\ref{fig:egy}. More generally, the diagram of the direct product of a chain 
$\{b_0\prec b_1\prec\cdots\prec b_m\}$ (to be placed on the bottom left boundary) and a chain 
$\{c_0\prec c_1\prec\cdots\prec c_n\}$ 
(to be placed at the bottom right boundary) is also called a \emph{grid diagram of type} $m\times n$. 
For $i\in\set{1,\ldots,m}$ and $j\in\set{1,\ldots,n}$, let
\[\ijcell i j= \set{(b_{i-1},c_{j-1}), (b_{i},c_{j-1}),(b_{i-1},c_{j}),(b_{i},c_{j})};
\]
this sublattice (and subdiagram) is called a \emph{$4$-cell}. The smallest join-congruence of $G$ that collapses the top boundary $\set{(b_{i},c_{j-1}),(b_{i-1},c_{j}),(b_{i},c_{j})}$ of this $4$-cell is denoted by $\cgjcon{\ijcell i j}$. We recall the following statement from  G.~Cz\'edli~\cite[Corollary 22]{r:czg-mtx};
for $(i,j)=(2,3)$ see the thick edges in Figure~\ref{fig:egy} for an illustration.

\begin{lemma}\label{jIncRckL} Let $G$ be a grid diagram of type $m\times n$, and let $i\in\set{1,\ldots,m}$ and $j\in\set{1,\ldots,n}$. Denote $\cgjcon{\ijcell i j}$ by $\balpha$.  Then
\begin{enumeratei}
\item\label{jIncRckLa} 
the $\balpha$-block $\blockk{(b_{i},c_{j})}\balpha$ of $(b_i,c_j)$ is $\set{(b_{i},c_{j-1}),(b_{i-1},c_{j}),(b_{i},c_{j})}$;
\item\label{jIncRckLb}
$\set{(b_s,c_{j-1}),(b_s,c_{j})}$ for $s>i$ and $\set{(b_{i-1},c_t),(b_{i},c_t)}$ for $t>j$ are the two-element blocks of $\balpha$;
\item\label{jIncRckLc} the rest of $\balpha$-blocks are singletons, and $\balpha$ is cover-preserving.
\end{enumeratei}    
\end{lemma}

The following description of join of join-congruences is borrowed from G.~Cz\'edli and E.\,T.~Schmidt~\cite[Lemma 11]{r:czg-sch-survey}.

\begin{lemma}\label{LMdesCrjniN}
Let $\bbeta_i$, $i\in I$, be join-congruences of a join-semilattice, and let $ u, v\in G$. Then $( u, v)\in\bigvee_{i\in I}\bbeta_i$ if{}f  there is a $k\in\mathbb N_0=\set{0,1,2,\ldots}$ and  there are elements 
\[ u= z_0\leq z_1\leq \cdots \leq  z_k= w_k\geq  w_{k-1}\geq\cdots\geq  w_0= v\]
such that $\set{( z_{j-1}, z_j),( w_{j-1}, w_j)}\subseteq \bigcup_{i\in I}\bbeta_i$ for $j\in\set{1,\ldots,k}$.
\end{lemma}

Combining Lemmas~\ref{jIncRckL} and  \ref{LMdesCrjniN} we easily obtain the following corollary, which is implicit in  G.~Cz\'edli~\cite{r:czg-mtx}.

\begin{corollary}\label{hTdszkeCcj} Let $G$ be a grid diagram of type $m\times n$,  and let $k\in\mathbb N=\set{1,2,\ldots}$. Assume that $1\leq i_1<\cdots <i_k\leq m$ and that $j_1,\ldots, j_k$ are pairwise distinct elements of $\set{1,\ldots n}$. Consider the join-congruence
$\bbeta=\bigvee_{s=1}^k \cgjcon{\ijcell {i_s}{j_s}}$. Then $\bbeta$ is cover-preserving, and it is described by the following rules.
\begin{enumeratei}
\item\label{hTdszkeCcja} $\bigl( (b_i,c_j),(b_s,c_t)\bigr)\in\bbeta$ if{f} $\set{(b_i,c_j),(b_s,c_t)}\subseteq \blockk{(b_i\vee b_s,c_j\vee c_t)}{\bbeta}$;
\item\label{hTdszkeCcjb} for $0\leq r<s$, $\bigl((b_r,c_t), (b_s,c_t)\bigr)\in\bbeta$ if{f} for each $x\in\set{r+1,\ldots,s}$ there is a $($unique$)$ $p\in\set{1,\ldots,k}$ such that  $x=i_p$ and $j_p\leq t$;
\item\label{hTdszkeCcjc} for $0\leq r<s$, $\bigl((b_t,c_r), (b_t,c_s)\bigr)\in\bbeta$ if{f} for each $x\in\set{r+1,\ldots,s}$ there is a $($unique$)$ $p\in\set{1,\ldots,k}$ such that  $x=j_p$ and $i_p\leq t$.
\end{enumeratei}
\end{corollary}

In Figure~\ref{fig:egy}, this statement is illustrated for $m=n=4$ and $k=4$ so that the non-singleton blocks of $\bbeta$ are indicated by dotted lines and  the $\ijcell {i_s}{j_s}$, $1\leq s\leq 4$, are the grey cells, that is,  
\begin{equation}\label{mTapermI}
\begin{pmatrix}i_1 &\dots& i_4\cr
j_1 &\dots& j_4
\end{pmatrix}= \begin{pmatrix}1&2&3&4\cr
4&3&1&2
\end{pmatrix}\text.
\end{equation}
Corollary~\ref{hTdszkeCcj} is also illustrated by Figure~\ref{fig:ket} for $m=n=8$ and $k=4$ where
\begin{equation}\label{pRtLmtShB}
\begin{pmatrix}i_1 &\dots& i_4\cr
j_1 &\dots& j_4
\end{pmatrix}= \begin{pmatrix}4&5&7&8\cr
4&8&1&2
\end{pmatrix}
\end{equation}
(only the dark grey $4$-cells are considered, the light grey one should be disregarded).

Now, we consider the grid diagram $G$ (of type $h\times h$) associated with $D$ again. The map $\phi\colon G\to D$, $(x,y)\mapsto x\vee y$ is a surjective  join-homomorphism. By G.~Cz\'edli and E.\,T.~Schmidt~\cite[proof of Corollary 2]{r:czg-sch-howtoderive}, $\phi$ is \emph{cover-preserving}, that is, if $ a, b\in G$ and $ a\preceq  b$, then $\phi( a)\preceq\phi( b)$.
Thus its kernel, $\balpha=\{( a, b): \phi( a)=\phi( b)\}$ is a so-called \emph{cover-preserving join-congruence} by definition, see \cite{r:czg-sch-howtoderive}. 
 If the $\balpha$-block 
$\blockk{(b_i,c_j)}{\balpha}$ includes $\set{(b_{i},c_{j-1}),(b_{i-1},c_{j})}$ but
$(b_{i-1},c_{j-1})\notin \blockk{(b_i,c_j)}{\balpha}$,  then $\ijcell i j$ is called a \emph{source cell} of $\balpha$. 
In Figure~\ref{fig:egy}, the source cells of $\balpha=\Ker\phi$ are the grey ones.
The set of these source cells is denoted by $\scells\balpha$. 
With $D$, we associate a relation $\pi(D)$ (which turns out to be a permutation, see \eqref{mTapermI} for Figure~\ref{fig:egy})  as follows:
\begin{equation}\label{fkintropi}\pi(D)=\bigset{(i,j)\in \set{1,\ldots,h}^2: \ijcell i j\in \scells\balpha} \text.
\end{equation}
Remember that similar diagrams are considered equal. The following result can be extracted from 
 G. Cz\'edli and E.\,T. Schmidt~\cite{czgschperm} and Lemma~\ref{LMdesCrjniN}. (Some parts that are formulated in \cite{czgschperm} for lattices rather than diagrams are explicitly given in  G.~Cz\'edli and G.~Gr\"atzer~\cite{ggwltsta}.) 

\begin{lemma}\label{lMsckLs} Let $D$ be a slim, semimodular, planar diagram of length $h$, and let $G$,  $\phi\colon G\to D$,  $\balpha=\Ker\phi$, and $\pi=\pi(D)$  be as above.
\begin{enumeratei}
\item\label{lMsckLsa} $\pi$ is a permutation on $\set{1,\ldots,h}$.
\item\label{lMsckLsb} $\balpha=\bigvee_{i=1}^h  \cgjcon{\ijcell i{\pi(i)}}$. 
\item\label{lMsckLsc} The mapping $D\mapsto \pi(D)$ is a bijection from the set of slim, semimodular diagrams of length $h$ to the set $S_h$ of permutations acting on $\set{1,\ldots,h}$.
\end{enumeratei}
\end{lemma}

\subsection*{Permutations determine the size}
For a permutation $\sigma\in S_h$, the number $|\{(\sigma(i),\sigma(j)): i<j$  and $\sigma(i)>\sigma(j)\}|$ of inversions of $\sigma$ is denoted by $\inv\sigma$. The same notation applies for \emph{partial permutations} (that is, bijections between two subsets of $\set{1,\ldots,h}$), only we have to stipulate that both $\sigma(i)$ and $\sigma(j)$ should be defined. For example, if $\sigma$ is the partial permutation given in 
\eqref{pRtLmtShB}, then $\inv\sigma=4$. 
The size $|D|$ of a diagram $D$ is the number of elements of the lattice it determines. 
A crucial step of the paper is represented by the following statement.

\begin{proposition}\label{crLLtThma} With the assumptions of Corollary~\ref{hTdszkeCcj}, let $K$ be the lattice determined by $G$, and let $\tau$ denote the partial permutation
$\begin{pmatrix}i_1 &\dots& i_k\cr
j_1 &\dots& j_k
\end{pmatrix}$. Then 
\begin{equation}\label{prizTfGx}
|K/\bbeta|= (m+1)(n+1)+\inv\tau -k(m+n+2)+\sum_{s=1}^k(i_s+j_s)\text.
\end{equation}
\end{proposition}

\begin{proof}
We prove \eqref{prizTfGx} by induction on $k$. The case $k=0$ is obvious since then $\bbeta$ is the least join-congruence, $\inv\tau=0$, and $K/\bbeta\cong K$. Hence we assume that $k>0$ and the lemma holds for all smaller values. 
We let 
\[\sigma = \begin{pmatrix}
i_2&\ldots&i_k\cr
j_2&\ldots&j_k
\end{pmatrix}\text.
\]   
The situation is depicted in Figure~\ref{fig:ket}, where $k=5$,  
\[\tau=\begin{pmatrix}
2&4&5&7&8\cr
5&4&8&1&2
\end{pmatrix}
\]  is given by the four dark grey $4$-cells plus the light grey $4$-cell $\ijcell 25$, and  $\sigma$ is given by the four the dark grey $4$-cells. (Note that $\sigma$ is the partial permutation 
in \eqref{pRtLmtShB} but now  the subscripts are shifted by 1.) Let 
\[\bbeta'=\bigvee_{s=2}^k  \cgjcon{\ijcell {i_s}{j_s}} = \bigvee_{s=2}^k  \cgjcon{\ijcell {i_s}{\sigma(i_s)}}\text;
\]
its blocks are indicated by dotted lines in Figure~\ref{fig:ket}.
By the induction hypothesis, the number of $\bbeta'$-blocks is
\begin{equation}\label{prikmineyTfy}
|K/\bbeta'|= (m+1)(n+1)+\inv\sigma -(k-1)(m+n+2)+\sum_{s=2}^k(i_s+j_s)\text.
\end{equation}
Consider the following elements (see them enlarged in Figure~\ref{fig:ket}):
\[\begin{matrix}
u=(b_{i_1},c_{j_1-1}),&v=(b_{i_1},c_{j_1}),&  w=(b_{i_1-1},c_{j_1}),& z=(b_{i_1},c_0),\cr
u'=(b_{m},c_{j_1-1}),& v'=(b_{m},c_{j_1}),&
v''=(b_{i_1},c_{n}),& w''=(b_{i_1-1},c_{n})\text.
\end{matrix}
\]
The restriction of $\bbeta'$ to an interval $I$ will be denoted by $\restrict {\bbeta'\rspace}I$, and $\bomega_I$ stands for the smallest equivalence on $I$. 
Since no dark grey $4$-cell occurs in the interval $[0,v'']$, Corollary~\ref{hTdszkeCcj} gives that $\restrict{\bbeta'\rspace}{[0,v'']}=\bomega_{[0,v'']}$. 
Similarly,  there is no $t$ such that  $\bigl((b_t, c_{j_1-1}), (b_t, c_{j_1})\bigr)\in \bbeta'$. Let $\bgamma_u$ be the join-congruence of $[z,u']$ defined by
\begin{equation}\label{gMmau}\bgamma_u=
\bigvee\bigset{ \cgjcon{\ijcell {i_s}{j_s}}: 1<s\leq k,\,\, j_s<j_1 } ; 
\end{equation}
it is the smallest join-congruence of $[z,u']$ that collapses the top boundaries of the dark grey $4$-cells in $[z,u']$. 
We conclude that $\bdelta=\bomega_{[0,v'']}\cup\bomega_{[u,v']}\cup \bgamma_u \cup [v,1]^2 $, which is clearly a join-congruence of $K$, 
 includes $\bbeta'$. Thus $\bgamma_u=\restrict{\bbeta'\rspace}{[z,u']}$. If \eqref{gMmau} is understood in the interval $[z,v']$, then it defines a join-congruence $\bgamma_v$ of $[z,v']$, and we similarly obtain that $\bgamma_v=\restrict{\bbeta'\rspace}{[z,v']}$. 
The previous two equalities clearly yield that $\restrict{\bbeta'\rspace}{[u,u']} = \restrict{\bgamma_u\rspace}{[u,u']}$ and $\restrict{\bbeta'\rspace}{[v,v']} = \restrict{\bgamma_v\rspace}{[v,v']}$.
Applying Corollary~\ref{hTdszkeCcj} to $[z,u']$ and to $[z,v']$,  we obtain that $\restrict{\bbeta'\rspace}{[u,u']}$ partitions $[u,u']$ to
$m+1-i_1-q$ blocks and that $\restrict{\bbeta'\rspace}{[v,v']}$ partitions $[v,v']$ to $m+1-i_1-q$  blocks, where
\[q =  |\set{s: 1<s\leq k,\,\, j_s<j_1 }|,\]
which is the number of dark grey $4$-cells in $[z,u']$ (and also in $[z,v']$). 
We also obtain from Corollary~\ref{hTdszkeCcj} that the above-mentioned blocks are ``positioned in parallel'', that is, for $x,y\in [u,u']$, we have $(x,y)\in\bbeta'$ if{f} $(x\vee v,y\vee v)\in \bbeta'$.

We know that $\bbeta=\bbeta'\vee \cgjcon{\ijcell {i_1}{j_1}}$ in the lattice of join-congruences of $K$ and also in the lattice of equivalences of $K$. The blocks of $\cgjcon{\ijcell {i_1}{j_1}}$ are given by Lemma~\ref{jIncRckL}; they are indicated by thick lines in Figure~\ref{fig:ket}.
Since $\bbeta'\subseteq \bdelta$, each element of $[w,w']$ belongs to a singleton $\bbeta'$-blocks. There are $n+1-j_1$ such (singleton) $\bbeta'$-blocks, and the northwest-southeast oriented thick edges merge them into other (not necessarily singleton) $\bbeta'$-blocks. Similarly, the northeast-southwest oriented thick edges merge  $q$ \ $\bbeta'$-blocks of $[z,u']$ to the respective  blocks in $[v,v']$.
Therefore, 
\begin{equation}\label{aiJfWhb}
|K/\bbeta|=|K/\bbeta'|-(m+1-i_1-q)-(n+1-j_1)\text.
\end{equation}
Since $q$ is the number of inversions with $j_1$, we have that $q=\inv\tau-\inv\sigma$. Combining this equation with \eqref{prikmineyTfy} and \eqref{aiJfWhb} 
we obtain the desired \eqref{prizTfGx}.
\end{proof}

\begin{proposition}\label{lMainvnoMla}
Let $D$ be a slim, semimodular, planar diagram, and let $\pi$ be the permutation associated with $D$ in \eqref{fkintropi}. Then $|D|=h+1+\inv\pi$.
\end{proposition}

\begin{proof}Let $L$ be the lattice determined by $D$. It follows from  Lemma~\ref{lMsckLs} and the Homomorphism Theorem that $|D|=|L|=|G/\balpha|$. Hence Proposition~\ref{crLLtThma} applies, and the substitution $(m,n,k,\sigma):=(h,h,h,\pi)$ clearly turns the  right side of \eqref{prizTfGx}  into $h+1+\inv\pi$.
\end{proof}

\subsection*{Permutations corresponding to slim, distributive lattices} For a planar diagram $D$ of a slim, semimodular lattice $L$, let $\PIntrv D$ denote the set of prime intervals of $L$, that is, the set of edges of $D$. The transitive reflexive closure of the relation
\[\bigset{\bigl([a,b],[c,d]\bigr):[a,b] \text{ and }[c,d]\text{ are opposite sides of a $4$-cell}  }
\]
is called \emph{prime projectivity}. It is an equivalence relation on $\PIntrv D$, and its blocks are called \emph{trajectories}. A trajectory can be visualized by its \emph{strip}, which is the set of $4$-cells determined by consecutive edges of the trajectory. 
For example, the strip from $[g_B,g_B']$ to $[h_B,h_B']$ in Figure~\ref{fig:har} is depicted in grey.  If only some consecutive edges of a trajectory are taken, then they determine a \emph{strip section}.  We recall the following statement from  G.~Cz\'edli and  E.\,T.~Schmidt~\cite[Lemma 2.8]{r:czg-sch-JH}.

\begin{lemma}\label{traJlemMa} Each trajectory of $D$ starts at a unique prime interval of $\leftb D$, and it goes to the right. First it goes upwards $($possibly in zero steps$)$, then it goes downwards $($possibly in
zero steps$)$, and finally it reaches a unique prime interval of $\rightb D$. 
In particular, once it is going down,
there is no further turn.
Trajectories never branch out.
\end{lemma}

Assume that $D$ is a slim, semimodular diagram with boundary chains \eqref{lrlchNps}. By Lemma~\ref{traJlemMa}, for each $i\in\set{1,\ldots,h}$ there is a unique $j\in \set{1,\ldots,h}$ such that the trajectory starting at $[b_{i-1},b_i]$ arrives at $[c_{j-1}, c_{j}]$. This defines a map $\hat\pi(D)\colon \set{1,\ldots,h}\to \set{1,\ldots,h}$, $i\mapsto j$.  For example, $\hat\pi(D)$ for Figure~\ref{fig:har} is
\begin{equation}\label{fFfighpRm}
\hat\pi(D)=\begin{pmatrix}
1&2&3&4&5&6&7&8\cr
2&7&6&4&1&8&3&5
\end{pmatrix}\text.
\end{equation}
This gives an alternative way to associate a permutation with $D$ using the following statement from G.~Cz\'edli and   E.\,T.~Schmidt~\cite{r:czg-sch-patch}.

\begin{lemma}\label{PiOthDflmA} For any planar, slim, semimodular diagram $D$, 
 $\hat\pi(D)$ equals $\pi(D)$ defined in \eqref{fkintropi}.
\end{lemma}

Let $\pi\in S_h$. We say that the permutation $\pi$ \emph{contains the 321 pattern} if there are $i<j<k\in\set{1,\ldots,h}$ such that $\pi(i)>\pi(j)>\pi(k)$. The distributivity of $D$ is characterized by the following statement.

\begin{proposition}\label{lMadisjLLp} Let $D$ be a slim, semimodular diagram, and let $\pi=\pi(D)$ denote the permutation associated with it. Then $D$ is distributive if{f} $\pi$ does not contain the 321 pattern.
\end{proposition}

\begin{figure}
\centerline
{\includegraphics[scale=1.0]{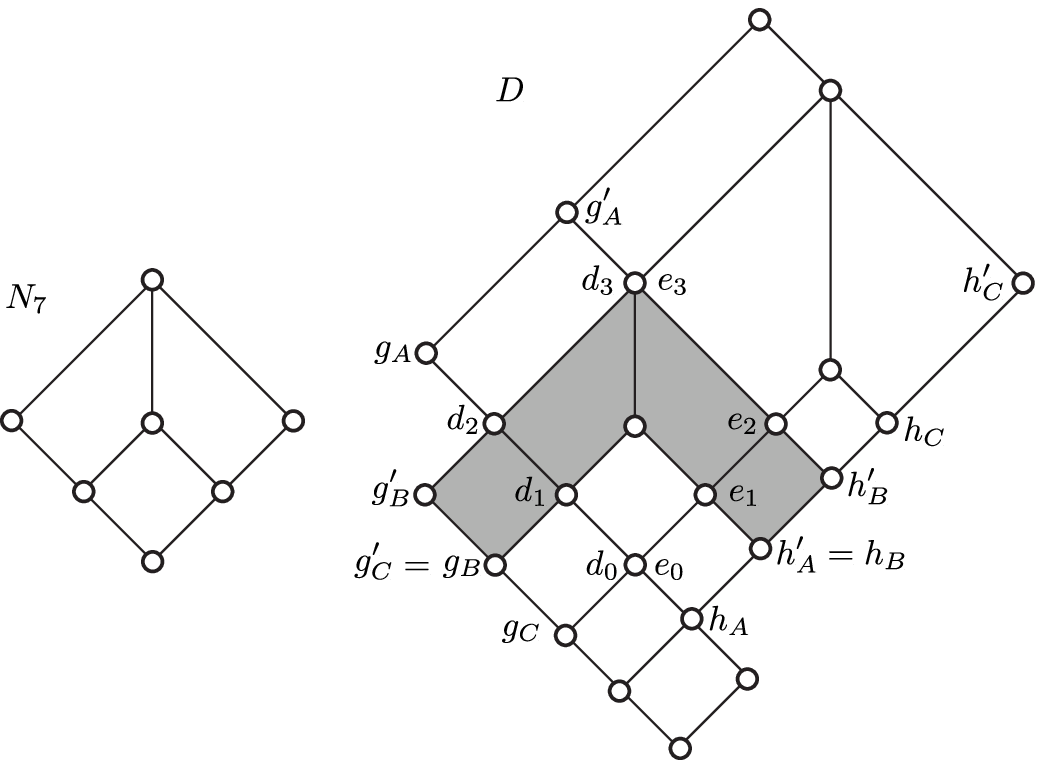}}
\caption{$\mySseven$ and a trajectory \label{fig:har}}
\end{figure}

\begin{proof} In virtue of Lemma~\ref{PiOthDflmA}, we work with $\pi=\hat\pi(D)$. In order to prove the necessity part of Proposition~\ref{lMadisjLLp}, we assume  that $D$ is not distributive. We obtain from G.~Cz\'edli and  E.\,T.~Schmidt~\cite[Lemma 15]{r:czg-sch-visual} that $D$ contains, as a cover-preserving sublattice, a copy of $\mySseven$, given in Figure~\ref{fig:har}. Let $\set{d_0\prec d_1\prec d_2\prec d_3}$ and $\set{e_0\prec e_1\prec e_2\prec e_3}$ be the left and right boundary chains, respectively, of a subdiagram of $D$ representing $\mySseven$, see Figure~\ref{fig:har}. Let $A,B,C$ denote the trajectories containing $[e_0,e_1]$, 
$[e_1,e_2]$, and $[e_2,e_3]$, respectively. The corresponding strip sections, starting at these edges and going to the right, are denoted by $A^{\ast}$, $B^{\ast}$, and $C^{\ast}$, respectively. 
Let us denote the last members of these trajectories by  $[h_A,h_A'], [h_B,h_B'],  
[h_C,h_C']\in \PIntrv{\rightb D}$, respectively. 
We claim that 
\begin{equation}\label{wRpbqSgxrghtB} h_A\prec h_A' \leq  h_B\prec h_B' \leq h_C\prec h_C'\text.
\end{equation}
Suppose for a contradiction that $h_A'\not\leq h_B$. Then $h_B'\leq h_A$ since $[h_A,h_A']\neq  [h_B,h_B']$  by Lemma~\ref{traJlemMa}, and $h_B'$ and $h_A$ belong to the chain $\rightb D$. Thus $A^{\ast}$ must cross $B^{\ast}$ at a $4$-cells such that $A^{\ast}$ crosses this $4$-cell upwards (that is, to the northeast). But this is impossible by Lemma~\ref{traJlemMa} since $A$ and thus $A^\ast$ went downwards previously at $[e_0,e_1]$.
A similar contradiction is obtained from $h_B'\not\leq h_C$ since $B^\ast$ goes downwards through $[e_1,e_2]$,  and thus it cannot cross a square upwards later. This proves \eqref{wRpbqSgxrghtB}.  

Next, let $[g_A,g_A'], [g_B,g_B'],  
[g_C,g_C']\in \PIntrv{\leftb D}$ denote the first edges of $A,B,C$, respectively. Since $[d_0,d_1]\in C$, $[d_1,d_2]\in B$, and $[d_2,d_3]\in A$, the left-right dual of the argument leading to \eqref{wRpbqSgxrghtB} yields that 
\begin{equation}\label{uRgmszsSgxrght} g_C\prec g_C' \leq  g_B\prec g_B' \leq g_A\prec g_A'\text.
\end{equation}
Therefore, in virtue of Lemma~\ref{PiOthDflmA},  \eqref{wRpbqSgxrghtB} together with \eqref{uRgmszsSgxrght} yields a 321 pattern in $\pi$.

Now, to prove the sufficiency part, assume that $D$ is distributive. Then it is dually slim by G.~Cz\'edli and  E.\,T.~Schmidt~\cite[Lemma 16]{r:czg-sch-visual}. Hence, by the dual of \cite[Lemma 16]{r:czg-sch-visual}, no element of $D$ has more than two lower covers. Thus each trajectory goes (entirely) either upwards,  or  downwards; that is, a trajectory cannot make a turn. 
Suppose for a contradiction that $\pi$ contains a 321 pattern. Then, like previously, we have trajectories $A,B,C$ such that \eqref{wRpbqSgxrghtB} and \eqref{uRgmszsSgxrght} holds. Any two of the corresponding strips must cross at a $4$-cell since their starting edges are in the opposite order as their ending edges are. Therefore any two of the three strips go to different directions, which is impossible since there are only two directions: upwards and downwards. This contradiction completes the proof.
\end{proof}

\subsection*{Permutations with the same lattice} To accomplish our goal, we have to know when two permutations determine the same slim, semimodular lattice. Below, we recall the necessary information and notation from G.~Cz\'edli and E.\,T.~Schmidt~\cite{czgschperm}, see also G.~Cz\'edli and G.~Gr\"atzer~\cite{ggwltsta} and  G.~Cz\'edli, L.~Ozsv\'art, and B.~Udvari \cite{r:czgolub}. 
Assume that $1 \leq u \leq v \leq h$ and $\pi\in S_h$. If $I = [u, v] = \set{i\in \mathbb N :  u \leq i \leq v}$ is nonempty and $[1, u - 1]$, $I$, and $[v + 1, h]$ are closed with respect to $\pi$, then $I$ is called a \emph{section} of $\pi$. Sections minimal with respect to set inclusion are called \emph{segments}. Let $\Seg\pi$ denote the set of all segments of $\pi$. For example, 
if $\pi=\hat\pi(D)$ from \eqref{fFfighpRm}, then $\pi$ has only one segment, $\set{1,\ldots,8}$. Another example is
\begin{equation}\label{eiKwGHsflIKG}
\pi=\begin{pmatrix}
1&2&3&4&5&6\cr
1&3&4&2&6&5
\end{pmatrix}\text{ with }\Seg\pi=\bigset{
\set1,\set{2,3,4},\set{5,6}}\text.
\end{equation}
For a subset $A$ of $\set{1,\ldots,h}$, let $\restrict \pi A$ denote the restriction of $\pi$ to $A$. The set of $A\to A$ permutations is denoted by $S_A$. 
Notice that $\Seg\tau$ also makes sense for $\tau\in S_A$ since the natural order of $\set{1,\ldots,h}$ is automatically restricted to $A$. If $A\in\Seg \pi$, then $\restrict \pi A\in S_A$ and $\restrict \pi {\set{1,\ldots,h}-A}\in S_{\set{1,\ldots,h}-A}$. The unique ${I_1}\in\Seg\pi$ with $1\in {I_1}$ is the \emph{initial segment} of $\pi$. 
We adopt the following terminology:
\begin{equation}\label{termheadBody}
\begin{aligned}
 \head\pi&=\restrict \pi {I_1}\in S_{I_1} \text{ is the \emph{head} of $\pi$,}\cr 
\body\pi&=\restrict \pi {\set{1,\ldots,h}-{I_1}}\in S_{\set{1,\ldots,h}-{I_1}}\text{ is the \emph{body} of $\pi$.}
\end{aligned}
\end{equation}  
Note that $\body\pi$ can be the empty permutation acting on $\varnothing$. Clearly, the pair $\pdecomp\pi$  determines $\pi$;
 however, the two components of the pair $\pdecomp\pi$ are not arbitrary. We say that $\pi\in S_h$ is \emph{irreducible}, if its initial segment is $\set{1,\ldots,h}$. Note that $\pi$ is irreducible  if{f} $\head\pi=\pi$ or, equivalently, if{f} $\body\pi=\varnothing$.  Clearly, if $I=[1,u]$ is a \emph{nonempty initial interval} of $\set{1,\ldots,h}$, that is, if $1\leq u\leq h$, and, in addition, $\sigma\in S_I$, and $\tau\in S_{\set{1,\ldots,h}-I}$, then 
$(\sigma,\tau)$ coincides with  $\pdecomp\pi$ for some $\pi\in S_h$ if{f} $\sigma\in S_I$ is irreducible. For $\pi\in S_h$, the \emph{degree} of $\pi$ is $h$. 
We define the \emph{block}$\,$ $\pblock\pi$  of
$\pi$ by induction on the degree of $\pi$ as follows. 
If $\pi$ is irreducible, then we let $\pblock\pi=\set{\pi,\pi^{-1}}$. Otherwise, let
$\pblock\pi=\bigset{\sigma: \head\sigma \in \set{\head\pi, \head\pi^{-1}}\text{ and }  
\body\sigma \in \pblock {\body\pi} }$. 
%
%
Notice that 
\begin{align}
\simfactor{S_h} &= \set{\pblock\pi:\pi\in S_h}\text{ and, for every $\pi\in S_h$,}\cr
\pblock\pi&=\set{\sigma: \pblock{\head \sigma}=\pblock{\head\pi}\text{ and }\pblock{\body \sigma}=\pblock{\body\pi}}\text.\label{iGxzgahEw}
\end{align}
is the partition on $S_h$ associated with the so-called ``sectionally inverse or equal'' relation introduced in
G.~Cz\'edli and  E.\,T.~Schmidt~\cite{czgschperm}. 
It is well-known from H.\,A.~Rothe~\cite{rothe}, see also D.\,E.~Knuth~\cite{knuth} or one can prove it easily, that $\inv\sigma=\inv{\sigma^{-1}}$. This implies that $\inv\sigma=\inv\pi$ for every $\sigma\in\pblock\pi$. Hence we can define $\inv{\pblock\pi}$ by the equation $\inv{\pblock\pi}=\inv\pi$. 
The following statement is also taken from \cite{czgschperm}, see also part \eqref{lMsckLsc} of Lemma~\ref{lMsckLs} here.

\begin{lemma}\label{sicIweim} Let $D$ and $E$ be slim, semimodular, planar diagrams.  Then $D$ and $E$ determine isomorphic lattices if{f} $\,\pblock{\pi(D)}=\pblock{\pi(E)}$. 
\end{lemma}

\section{Counting}\label{countsection} 
\subsection*{Slim, semimodular lattices}
We introduce the following notation.
\begin{align*}
\kern 8pt&\begin{aligned}
\Pernik h k&=\set{\pi\in S_h: \inv \pi=k}\text{, }\cr
\Tpernik h k&=\set{\pi\in\Pernik h k: \pi^2=\id}\text{, }
\end{aligned}
\begin{aligned}
\Ipernik h k&=\set{\pi\in\Pernik h k: \pi\text{ is irreducible}},\cr
\Bpernik h k&=\set{\pblock\pi: \pi \in\Pernik h k},
\end{aligned}
\cr
&\kern-3pt \Bhpernik s t h k=\set{\pblock\pi: \pi \in\Pernik h k\text{, }\head\pi\in S_s \text{, and } \inv{\head\pi}=t},
\cr&\kern 2pt \Itpernik h k =\set{\pi\in \Tpernik h k: \pi\text{ is irreducible}}\text.
\end{align*}
Here $\Perniksign$ and $\Tperniksign$ comes from ``permutation'' and ``involution''. Their parameters denote the length of permutations and the number of inversions, while ${}^\sim$ and $\kern 3pt\widehat{}\kern 3pt$ stand for blocks and irreducibility, respectively.   
The sizes of these sets are denoted by the corresponding lower case letters; for example, $\bpernik h k=|\Bpernik h k|$.
(Note that, as opposed to us, the literature denotes  $|\Pernik h k|$ usually by $I_h(k)$ rather than $\pernik h k$.) The binary function  $\perniksign$ is well-studied. Let
\begin{equation}\label{geNufcit}
G_h(x) = \sum_{j=0}^{\binomial h2}\pernik hj x^j
\end{equation}
denote its generating function. 
We recall the following result of 
O.\ Rodriguez~\cite{rodrigez} and  Muir~\cite{muir}, see also D.\,E.~Knuth~\cite[p.~15]{knuth}, or
B.\,H.~Margolius~\cite{margolius}, or M.~B\'ona~\cite[Theorem 2.3]{bona}.

\begin{lemma}\label{rrodrigE}
$\displaystyle{G_h(x)=\prod_{j=1}^h \sum_{t=0}^{j-1}x^t  = \prod_{j=1}^h \frac{1-x^j}{1-x}  }$.
\end{lemma}

We are now in the position to formulate the following theorem.

\begin{theorem}\label{thmmaIn}
The number $\numssl n$ of slim, semimodular lattices of size $n$ is determined by Lemma~\ref{rrodrigE} together with   
the following $($recursive$)$ formulas 
{\allowdisplaybreaks
\begin{align}
\numssl n&=\sum_{h=0}^{n-1} \bpernik h{n-h-1},  \label{thmmaIna} \\
\bpernik h k&=\frac 12\cdot\sum_{s=1}^h \sum_{t=0}^k   \bigl(\ipernik st + \itpernik st\bigr) \cdot \bpernik {h-s}{k-t}, \label{thmmaInb}\\
\ipernik h k  &= \pernik h k - \sum_{s=1}^{h-1}\sum_{t=0}^k \ipernik st \cdot \pernik{h-s}{k-t}, \label{thmmaInc}\\
\tpernik h k &=
\tpernik {h-1}k+\sum_{s=2}^h \tpernik {h-2}{k-2s+3},    \label{thmmaInd}\\
\itpernik h k &=  \tpernik h k -   \sum_{s=1}^{h-1}\sum_{t=0}^k \itpernik st \cdot \tpernik{h-s}{k-t} \label{thmmaIne}
\end{align}
for $n, h\in \mathbb N$ and $k\in\mathbb N_0$, and with the initial values 
\begin{align*}
\bpernik h0&=\pernik h0=\tpernik h0=1=\ipernik 10=\itpernik 10 \,\text{ for }\,h\in\mathbb N_0,\\
\bpernik h k=\pernik h k&=\ipernik h k=\tpernik h k=\itpernik h k=0\,\text{ if }\, k>\binomial h2\text{ or } \set{h,k}\not\subseteq \mathbb N_0,\\
\ipernik h 0&=\itpernik h0=0,\text{ if }\,h>1
\text.
\end{align*}
}
\end{theorem}

Notice that $\binomial h k=0$ if $k>h$.
Clearly, together with the initial values, \eqref{thmmaInd} determines the function $\tperniksign$,  \eqref{thmmaIne} gives the function $\itperniksign$, we can evaluate the function $\perniksign$ based on  Lemma~\ref{rrodrigE} and 
\eqref{geNufcit}, then \eqref{thmmaInc}
determines the function $\iperniksign$, \eqref{thmmaInb} yields $\bperniksign$, and, finally, \eqref{thmmaIna} yields $\numssl n$.

\begin{proof}[Proof of Theorem~\ref{thmmaIn}] By Lemma~\ref{sicIweim}, we have to count the blocks $\pblock p$ that give rise to $n$-element lattices. The initial values are obvious.

If $\pi\in S_h$, then  $\inv{\pblock p}= \inv \pi$ equals  $n-h-1$ by Proposition~\ref{lMainvnoMla}. This implies \eqref{thmmaIna}.

Next, $\pblock{\head\pi}$ is a singleton if $\head\pi^2=\id$, and it is two-element otherwise.  Thus, by \eqref{iGxzgahEw}, the number of blocks $\pblock\pi\in\Bpernik h k$ with $\head\pi^2=\id$ is 
\begin{equation}\label{nszeGwinea}
\sum_{s=1}^h \sum_{t=0}^k  \itpernik st \cdot \bpernik {h-s}{k-t}\text.
\end{equation}
Similarly, the  number of blocks $\pblock\pi\in\Bpernik h k$ with $\head\pi^2\neq\id$ is
\begin{equation}\label{nszeGwineb}
\sum_{s=1}^h \sum_{t=0}^k  \frac 12\cdot \bigl(\ipernik st -\itpernik st\bigr) \cdot \bpernik {h-s}{k-t}\text.
\end{equation}
Forming the sum of \eqref{nszeGwinea} and \eqref{nszeGwineb}, we obtain  \eqref{thmmaInb}.

The subtrahend on the right of   \eqref{thmmaInc} is the number of the reducible members of $\Pernik h k$. This implies \eqref{thmmaInc}. 

For $\pi\in \Tpernik h k$, let $s=\pi(1)$. There are exactly $\tpernik {h-1}k$ many such $\pi$ with $s=1$; this gives the first summand in \eqref{thmmaInd}. Next, assume that $s>1$, and note that $\pi(s)=1$ since $\pi^2=\id$. Then, in the second row of the matrix
\[\begin{pmatrix}
1&2&\dots&s-1&s&s+1&\dots&h\cr
s&\pi(2)&\dots&\pi(s-1)&1&\pi(s+1)&\dots&\pi(h)
\end{pmatrix},
\]
there are $s-1$ inversions of the form $(x,1)$, $s-2$ inversions of the form $(s,y)$ with $y\neq 1$, and we also have the inversions of $\sigma=\restrict\pi{\set{1,\ldots,h}-\set{1,s}}$. Therefore, $\sigma$ has $k-(s-1+s-2)$ inversions, whence $\sigma$   can be selected in $\tpernik {h-2}{k-2s+3}$ ways. This explains the second part of \eqref{thmmaInd}, completing the proof of equation \eqref{thmmaInd}.

Finally, the argument for 
\eqref{thmmaIne} is essentially the same as that for \eqref{thmmaInc} since the 
subtrahend in \eqref{thmmaIne} is the number of reducible members of $\Tpernik h k$.
\end{proof}

\subsection*{Slim, semimodular diagrams}
Due to Lemma~\ref{lMsckLs}\eqref{lMsckLsc}, the first part of the previous proof for  \eqref{thmmaIna} clearly yields the following statement. Based on Lemma~\ref{rrodrigE}, it  gives an effective way to count the diagrams in question.

\begin{proposition}\label{prodzSkW} Up to similarity, the number $\numssd n$ of planar, slim, semimodular lattice diagrams with $n$ elements is 
\[\numssd n=\sum_{h=0}^{n-1} \pernik h{n-h-1}\text.
\] 
\end{proposition}

\begin{proof} If $\pi\in S_h$ determines an $n$-element diagram, then  $\inv \pi$ equals  $n-h-1$ by Proposition~\ref{lMainvnoMla}. This together with Lemma~\ref{lMsckLs}\eqref{lMsckLsc} implies our statement.
\end{proof}

\subsection*{Slim distributive diagrams}
As opposed to the previous statement, we are going to enumerate these diagrams of a given  length rather than a given size.
Let 
$C_h=(h+1)^{-1}\cdot\dsty{\binomial {2h}h}$ denote the $h$-th \emph{Catalan number}, see, for example, M.~B\'ona~\cite{bona}. 

\begin{proposition}\label{pRdDlNgTnh} Up to similarity, 
the number  of planar, slim, distributive lattice diagrams of length $h$  is $C_h$. 
\end{proposition}

\begin{proof} By Lemma~\ref{lMsckLs}\eqref{lMsckLsc} and Proposition~\ref{lMadisjLLp}, we need the number of permutations in $S_h$ that do not contain the pattern $321$. This number is $C_h$ by M.~B\'ona~\cite[Corollary 4.7]{bona}.
\end{proof}

\subsection{Calculations with Computer Algebra}\label{SeCtcompscc}
It follows easily from Theorem~\ref{thmmaIn} that $\numssl 1=1$, $\numssl 2=1$, $\numssl 3=1$,$\numssl 4=2$, 
 $\numssl 5=3$, $\numssl 6=5$, $\numssl 7=9$, $\numssl 8=16$, $\numssl 9=29$, and these values can easily be checked by listing the corresponding lattices. One can use computer algebra to obtain, say, $\numssl {20}=33\,701$,  $\numssl {30}=25\,051\,415$, and  $\numssl {40}= 19\,057\,278\,911$. In a typical personal computer (with parameters Intel(R) Core(TM)2 Duo CPU E8400, 3.00 GHz, 1.98 GHz, 3.25 GB RAM), one can even compute 
\[ \numssl {50}=14\,546\,017\,036\,127
\]
in three hours. This indicates that semimodularity together with slimness is a strong assumption since  
it took several days on a parallel supercomputer to count \emph{all} 18-element lattices, see  J.~Heitzig and  J.~Reinhold~\cite{r:heitzigreinhold}.

%
%
%
%

\end{document}

%% file: slimsemmacros.tex
\newenvironment{enumeratei}{\begin{enumerate}[\upshape (i)]} {\end{enumerate}}
\numberwithin{equation}{section}
\newcommand \gyemant {$\pmb\lozenge$}
\theoremstyle{plain}
 \newtheorem{theorem}{Theorem}[section]

 \newtheorem{lemma}[theorem]{Lemma}
 \newtheorem{proposition}[theorem]{Proposition}

 \newtheorem{corollary}[theorem]{Corollary}
\theoremstyle{definition}

\theoremstyle{remark}

\theoremstyle{plain} 
\theoremstyle{definition} 
\theoremstyle{remark} 

\newcommand\semmi[1] {}
\newcommand\piros[1]{{\textcolor{red}{#1}}}

\newcommand\tudnivalo [1] {}  

\newcommand \temphint [1]{}

\newcommand \skinfo [1] {}

\newcommand \dsty [1] {\displaystyle{#1}}
\newcommand \ssty [1] {\scriptstyle{#1}}
\newcommand \sssty [1] {\scriptscriptstyle{#1}}
\newcommand \tbf[1] {\textbf{#1}}

\renewcommand\phi{\varphi}
\renewcommand\rho{\varrho}   
\renewcommand\epsilon{\varepsilon} 
\renewcommand \theta {\vartheta}   

\newcommand \balpha {{\boldsymbol{\alpha}}}
\newcommand \bbeta  {\boldsymbol{\beta}}
\newcommand \bgamma{\boldsymbol{\gamma}}
\newcommand \bdelta{\boldsymbol{\delta}}

\newcommand \bomega{\boldsymbol{\omega}}

\newcommand \Jir [1] {\textup{Ji}\,#1} 

\newcommand \binomial [2] {{{#1}\choose{#2}}}

\newcommand \mySseven {N_7}

\newcommand \leftb [1]  {\textup{BC}_{{\textup{l}}}(#1)} 
\newcommand \rightb [1] {\textup{BC}_{{\textup{r}}}(#1)} 

\newcommand \ilupphold [2] {\textup{sp}_{\sssty{\textup{left}}}^{\kern-2pt +\kern 2 pt #2}(#1)} 
\newcommand \irupphold [2] {\textup{sp}_{\sssty{\textup{right}}}^{\kern-2pt +\kern 2pt #2}(#1)} 

\newcommand \PIntrv [1] {\textup{PrInt}(#1)} 

\newcommand \smLatsign {\textup{SSL}}

\newcommand \multi [1] {{\ddot{#1}}}
\newcommand \multismLatsign {\textup{SS}\multi{\textup{L}}}

\newcommand \dsmLat[1] {\smLatsign{}^{\kern-1pt\delta\kern-1pt}(#1)}
\newcommand \dpsmLat[1] {\kern-1pt{\multismLatsign}{}^{\kern-1pt\delta\kern-1pt}(#1)}

\newcommand \dismLat[1] {\smLatsign{}^{\kern-1pt\delta\kern-1pt}(#1)\fentjelkong1}
\newcommand \dipsmLat[1] {{\multismLatsign}{}^{\kern-1pt\delta\kern-1pt}(#1)\fentjelkong1}
\newcommand \fentjelkong[1]{{}^{\kern- #1pt \mathord{\cong}}}
\newcommand \fentjelsim[1]{{}^{\kern- #1pt \mathord{\sim}}}

\newcommand \sampleclssign {\mathcal K}

\newcommand \dsampleclasps[1]  {\sampleclssign{}^{\kern0pt\delta}(#1)}
\newcommand \disampleclasps[1] {\sampleclssign{}^{\kern0pt\delta}(#1)\fentjelkong0}

\newcommand \LatCompSerSign {\textup{CSL}}
\newcommand \multiLatCompSerSign {\textup{CS}\multi{\textup{L}}}

\newcommand \dLcs  [1] {\LatCompSerSign{}^{\kern0pt\delta}(#1)}
\newcommand \dpLcs [1] {{}^{\kern0pt\delta}\kern-1pt\multi{\LatCompSerSign}(#1)}

\newcommand \diLcs [1] {\LatCompSerSign{}^{\kern0pt\delta}(#1)\fentjelkong1}
\newcommand \dipLcs[1] {{\multiLatCompSerSign}{}^{\kern0pt\delta}(#1)\fentjelkong1}
\newcommand \blockk [2] {#1/#2}

\newcommand \Seg [1] {\textup{Seg}(#1)}
\newcommand \head [1] {\textup{head}(#1)}
\newcommand \body [1] {\textup{body}(#1)}
\newcommand \pdecomp [1] {(\head{#1},\body{#1})}
\newcommand \pblock [1] {[#1]^{\kern-1pt\mathord{\sim}}}
\newcommand \simfactor [1] {#1/{\kern-1pt\mathord{\sim}}}
\newcommand \Perniksign {P}
\newcommand \Pernik [2] {\Perniksign (#1,#2)} 
\newcommand \perniksign {p}
\newcommand \pernik [2] {\perniksign (#1,#2)} 
\newcommand \Tperniksign{I}
\newcommand \Tpernik [2] {\Tperniksign(#1,#2)} 
\newcommand \tperniksign{i}
\newcommand \tpernik [2] {\tperniksign(#1,#2)} 
\newcommand \Bpernik [2] {P^{\kern-1pt\mathord{\sim}}(#1,#2)} 
\newcommand \bperniksign {p^{\kern-1pt\mathord{\sim}}} 
\newcommand \bpernik [2] {\bperniksign(#1,#2)} 
\newcommand \Bhpernik [4] {P^{\kern-1pt\mathord{\sim}}_{#1,#2}(#3,#4)} 
\newcommand \bhpernik [4] {p^{\kern-1pt\mathord{\sim}}_{#1,#2}(#3,#4)} 
\newcommand \Ipernik [2] {\widehat{\kern 1pt P}(#1,#2)} 
\newcommand \iperniksign{\widehat {\kern 1pt p} } 
\newcommand \ipernik [2] {\iperniksign(#1,#2)} 
\newcommand \Itpernik [2] {\widehat{\kern 1pt I}(#1,#2)} 
\newcommand \itperniksign {\widehat {\kern 1pt i}}
\newcommand \itpernik [2] {\itperniksign(#1,#2)} 

\newcommand \restrict [2] {{#1}\kern-1pt \rceil_{\kern-1pt #2}}
\newcommand\set [1]{\{#1\}}
\newcommand \bigset[1] {\bigl\{#1\bigr\}} 

\newcommand \ijcell [2] {\textup{cell}(#1,#2)}
\newcommand \scells [1]   {\textup{SCells}(#1)} 
\newcommand \celldn [1] {\textup{cell}_{\kern-0.3pt\sssty{\mathord{\lozenge}}}(#1)}
\newcommand \cellup [1] {\textup{cell}^{\kern-0.7pt\sssty{\mathord{\lozenge}}}(#1)}
\newcommand \cgjcon [1] {\textup{con}_{\vee}\kern-1pt(#1)} 

\newcommand \boper [1] {#1^{\kern-1pt\ssty{{\mathord=\kern-5pt\mathord{\parallel}}}}}

\newcommand \kernit {{\kern -6pt}}
\newcommand \tbu[1] {{ \phantom{\Big|} \kern -2pt \boxed{\tbf{#1}}  }}  
\newcommand \tbw[1] {{ \phantom{\Big|} \kern -2pt \boxed{\tbf{#1}}\kernit}}  
\newcommand \vakvec [1] {\vec{#1}}

\newcommand \secie {\boldsymbol{\rho}_{\kern -1pt e}^i}
 
\newcommand \numssl [1] {N_{\textup{ssl}}(#1)}
\newcommand \numssd [1] {N_{\textup{ssd}}(#1)}

\newcommand \bpbeta[1] {\boldsymbol{\beta}_{\kern-1pt #1}}

\newcommand \joing {\mathrel{\mathord\vee_{\kern -2pt G}}}
\newcommand \joinl {\mathrel{\mathord\vee_{\kern -2pt L}}}
\newcommand \joinvg {\mathrel{\mathord\vee_{\kern -2pt G'}}}
\newcommand \diai [1] {#1^{{\kern-0.7pt\natural}}}
\newcommand \diah [1] {#1{}^{{\kern-0.7pt  \sssty{{\triangledown}} }}}

\newcommand \diag [1] {#1{}^{{\kern-0.7pt\ast}}}
\newcommand \diac [1] {(#1)^{{\kern-0.7pt\diamond}}}
\newcommand \diavarc [1] {#1^{{\kern-0.7pt\diamond}}}

\newcommand \semipatch  {{ \,\pmb{\mathcal H}\kern 0.5pt}} 
\newcommand \patch [1] { \pmb{\mathcal P}_{\kern-2pt\textup{max}}(#1)} 
\newcommand \aslim [1] {#1^{\kern-1pt \bullet}} 

\newcommand \Ker[1] {\textup{Ker}\,#1} 
\newcommand \cproj {\mathrel{ \mathord{\Rightarrow} \kern-7.5pt \mathord{\Rightarrow} }} 
\newcommand \cpreq {\mathrel{ \mathord{\Leftarrow}  \kern-7.5pt \mathord{\Leftrightarrow} \kern-7.5pt \mathord{\Rightarrow} }} 
\newcommand \uppers {\,{\buildrel{\sssty{\textup{up}}}\over \rightarrow}\kern-8pt\mathord{\rightarrow}\; } 
\newcommand \upperpa [1] {\,{\buildrel{\sssty{\textup{up}}}\over \rightarrow}\kern-8pt\mathord{\rightarrow}_{#1}\; } 
\newcommand \dnpers {\,{\buildrel{\sssty{\textup{dn}}}\over \rightarrow}\kern-8pt\mathord{\rightarrow}\; } 
\newcommand \dnperpa [1] {\,{\buildrel{\sssty{\textup{dn}}}\over \rightarrow}\kern-8pt\mathord{\rightarrow}_{#1}\; } 
\newcommand \id {\textup{id}} 
\newcommand \drestrict [2] {{#1}\rceil_{\kern-1pt #2}} 